\documentclass{amsart}

\usepackage{mathrsfs}
\usepackage{mathrsfs}
\usepackage{amsfonts}
\usepackage{amsfonts}
\usepackage{mathrsfs}
\usepackage{amsfonts}
\usepackage{amsfonts}
\usepackage{amsfonts}
\usepackage{mathrsfs}
\usepackage{amsfonts}
\usepackage{amssymb,amsmath}
\usepackage{amsthm}
\usepackage{amsfonts}

\newtheorem{theorem}{Theorem}[section]
\newtheorem{lemma}[theorem]{Lemma}

\theoremstyle{definition}
\newtheorem{definition}[theorem]{Definition}

\theoremstyle{remark}

\numberwithin{equation}{section}

\begin{document}

\title[Maximum Principle]
 {Maximum principles for a time-space fractional diffusion equation}

\author[J.X.Jia]{Junxiong Jia}
\address{Department of Mathematics,
Xi'an Jiaotong University,
 Xi'an
710049, China; }
\email{jjx323@mail.xjtu.edu.cn}


\author[K.X. Li]{Kexue Li}
\address{Department of Mathematics,
Xi'an Jiaotong University,
 Xi'an
710049, China;}
\email{kexueli@gmail.com}

\subjclass[2010]{35R11, 35B50, 34A08}



\keywords{Time-space fractional diffusion equation, Maximum principle, Fractional derivative}

\begin{abstract}
In this paper, we focus on maximum principles of a time-space fractional diffusion equation.
Maximum principles for classical solution and weak solution are all obtained by using properties
of the time fractional derivative operator and the fractional Laplace operator.
We deduce maximum principles for a full fractional diffusion equation, other than time-fractional and spatial-integer order
diffusion equations.
\end{abstract}

\maketitle


\section{Introduction}

In this paper, we focus on the following time-space fractional diffusion equation
\begin{align}\label{TimeSpaceEquation1}
\left\{\begin{aligned}
\partial_{t}^{\alpha}(u(x,t)-u_{0}(x)) + (-\Delta)^{\beta}u(x,t) & = f(x,t) \quad \text{in }\Omega\times[0,\infty), \\
u(x,t) & = 0 \quad\quad\quad\, \text{in }\mathbb{R}^{N}\backslash\Omega, \, t\geq 0, \\
u(x,0) & = u_{0}(x) \quad\,\, \text{in }\Omega,
\end{aligned}\right.
\end{align}
where $\Omega \subset\mathbb{R}^{N}$($N \geq 1$) is a bounded domain in $N$-dimensional space,
$\alpha, \beta \in (0,1)$ and $\partial_{t}^{\alpha}\cdot$ represents the Riemann-Liouville time-fractional derivative
defined as follow
\begin{align}\label{CaputoDefinition1}
\partial_{t}^{\alpha}v(t) := \frac{d}{dt}(g_{1-\alpha}*v(\cdot))(t),
\end{align}
with $g_{\gamma}(t) = \frac{t^{\gamma - 1}}{\Gamma(\gamma)}$ and ``$*$'' represents usual convolution operator.
The fractional Laplace operator could be defined as follow
\begin{align}\label{DefinitionFractionalLaplace}
(-\Delta)^{\beta}v(x) = c_{N,\beta}\int_{\mathbb{R}^{n}}
\frac{v(x) - v(y)}{|x - y|^{N+2\beta}}dy,
\end{align}
with $c_{N,\beta} = \frac{\beta 2^{2\beta}\Gamma(\frac{N+2\beta}{2})}{\pi^{N/2}\Gamma(1-\beta)}$ and $\Gamma(\cdot)$ represents the
usual Gamma function. For more properties about fractional Laplace operator, we refer to \cite{CPA:CPA20153}.

There are much research about maximum principles for equation (\ref{TimeSpaceEquation1})
when $\beta = 1$ \cite{Luchko2009218,al2014maximum},
which is a time fractional diffusion equation. In the fractional elliptic partial differential equation field,
there are also lots of research about maximum principles e.g. \cite{grecohopf}.
Recently, some maximum principles for the time fractional diffusion equations
have been applied to inverse source problems in \cite{Luchko2013Rundell}.

Although maximum principles are important tools, to the best of our knowledge, there are few results about
maximum principles for equation (\ref{TimeSpaceEquation1}) when $\alpha$, $\beta$ are both
non-integers. In this paper, we prove weak maximum principles for classical and weak solutions of full fractional
diffusion equation (\ref{TimeSpaceEquation1}) which may provide important tools for other research.

\textbf{Notations}: In the sequel, $W^{k,p}$ denotes the usual Sobolev spaces with derivative $k$ and Lebesgue exponent $p$;
$C^{k}$ denotes $k$ times differentiable function spaces.

\section{Fundamental Identity of the Time Fractional Derivative}

In the following proof, we need an important formula which could be found in \cite{Zacher2008137}
that is for a sufficiently smooth function $u$ on $(0,T)$ one has for a.e. $t\in (0,T)$,
\begin{align}\label{FundamentalIdentity}
\begin{split}
& H'(u(t))\frac{d}{dt}(k*u)(t) = \frac{d}{dt}(k*H(u))(t) + (-H(u(t))+H'(u(t))u(t))k(t) \\
& \quad
+ \int_{0}^{t}(H(u(t-s))-H(u(t))-H'(u(t))[u(t-s)-u(t)])\left(-\frac{dk(s)}{ds}\right)ds,
\end{split}
\end{align}
where $H\in C^{1}(\mathbb{R})$ and $k\in W^{1,1}([0,T])$. Denote $y^{+} = \max\{y,0\}$ and $y^{-}=\max\{-y,0\}$.
Now, taking $H(y)=\frac{1}{2}(y^{+})^{2}$,
for any function $u \in L^{2}([0,T])$, there will be a direct corollary of the above formula
\begin{align}\label{AppendexTimeFractional1}
u(t)^{+}\frac{d}{dt}(k*u)(t) \geq \frac{1}{2}\frac{d}{dt}(k*(u^{+})^{2}), \quad \text{a.e. }t\in(0,T).
\end{align}
Denote $v = -u$ and for $v$, we could also obtain
\begin{align}\label{AppendexTimeFractional2}
v(t)^{+}\frac{d}{dt}(k*v)(t) \geq \frac{1}{2}\frac{d}{dt}(k*(v^{+})^{2}), \quad \text{a.e. }t\in(0,T).
\end{align}
Now replacing $u$ back into (\ref{AppendexTimeFractional2}), we find that
\begin{align}\label{AppendexTimeFractional3}
u(t)^{-}\frac{d}{dt}(k*u)(t) \leq -\frac{1}{2}\frac{d}{dt}(k*(u^{-})^{2}), \quad \text{a.e. }t\in(0,T).
\end{align}

\section{Maximum Principle for Classical Solution}

In this section, firstly, let us introduce a lemma which could easily be obtained by using Theorem 1 in \cite{Luchko2009218}
and formula (1.20) in \cite{bajlekova2001fractional}.
\begin{lemma}\label{CaputoMaximumLemma}
Let a function $f \in W^{1,1}((0,T)) \cap C([0,T])$ attain its maximum (minimum) over the interval $[0,T]$ at the point
$\tau = t_{0}$, $t_{0}\in (0,T]$. Then the Riemann-Liouville fractional derivative of the function $f(\cdot) - f(0)$ is non-negative (non-positive)
at the point $t_{0}$ for any $\alpha$, $0<\alpha<1$,
\begin{align*}
\partial_{t}^{\alpha}(f(t_{0})-f(0)) \geq 0, \quad (\partial_{t}^{\alpha}(f(t_{0})-f(0)) \leq 0), \quad 0<\alpha < 1.
\end{align*}
\end{lemma}

\begin{definition}\label{BoundaryDefinition1}
Define the following concepts regarding the domain of the solution:
\begin{enumerate}
  \item $Q_{T} := \Omega\times(0,T) \subset \mathbb{R}^{N+1}$.
  \item Lateral boundary of $Q_{T}$: $\partial_{L}Q_{T}:=\partial\Omega\times[0,T]$.
  \item Parabolic boundary of $Q_{T}$: $\partial_{p}Q_{T} := (\Omega\times\{0\}) \cup \partial_{L}Q_{T}$.
\end{enumerate}
\end{definition}

\begin{theorem}\label{MaximumPrincipleClassicalSol1}
Let $\Omega \subset \mathbb{R}^{N}$ to be a bounded domain, and let $u(x,t)$ be a function that is $C^{2}$ in $x$
and $C^{1}$ in $t$ for $(x,t)\in \Omega\times(0,T)$, and continuous in both $x$ and $t$ for $(x,t) \in \bar{\Omega}\times[0,T]$;
and $u$ is a solution of equation (\ref{TimeSpaceEquation1}) with $f \geq 0$ in $\bar{Q}_{T}$, and
$u_{0} \geq 0$ in $\Omega$. Then $u \geq 0$ in $\bar{Q}_{T}$.
\end{theorem}
\begin{proof}
Consider $0 < T' < T$, and $\bar{Q}_{T'}$, and let us argue by contradiction.
Assume $u < 0$ somewhere in $\bar{Q}_{T'}$. Because $u \in C(\bar{Q}_{T'})$, and $\bar{Q}_{T'}$ compact,
there exist $(x_{0},t_{0}) \in \bar{Q}_{T'}$ such that $u(x_{0},t_{0}) = \min_{\bar{Q}_{T'}}u < 0$.
Since $u \geq 0$ in $\partial_{p}\bar{Q}_{T'} \subset \partial_{p}\bar{Q}_{T}$,
we have $(x_{0},t_{0}) \notin \partial_{p}\bar{Q}_{T'}$.

No matter $(x_{0},t_{0}) \in Q_{T'}$ is a minimum or $(x_{0},t_{0}) \in \Omega\times\{T'\}$ is a minimum,
we know that $\partial_{t}^{\alpha}(u(x_{0},t_{0}) - u(x_{0},0)) \leq 0$ from Lemma \ref{CaputoMaximumLemma}.
Because $u(\cdot,t_{0}) \in C^{2}(\Omega)\cap C(\bar{\Omega})$ and is zero outside the domain and
$u$ attains minimum at $(x_{0},t_{0})$, we have
\begin{align}\label{ClassicalMaximumProof1}
(-\Delta)^{\beta}u(x_{0},t_{0}) = c_{n,\beta}\int_{\mathbb{R}^{n}}
\frac{u(x_{0},t_{0}) - u(x,t_{0})}{|x_{0} - x|^{N+2\beta}}dx \leq 0.
\end{align}
If $(-\Delta)^{\beta}u(x_{0},t_{0}) = 0$, then $u(\cdot,t_{0}) = 0$,
which is a contradiction with $u(x_{0},t_{0}) < 0$, therefore $(-\Delta)^{\beta}u(x_{0},t_{0}) < 0$.
But, we have $0 \leq f(x,t) = \partial_{t}^{\alpha}(u(x_{0},t_{0}) - u(x_{0},0)) + (-\Delta)^{\beta}u(x_{0},t_{0}) < 0$.
It is a contradiction. Therefore, $u\geq 0$ in $Q_{T'}$. Now we obtain
$u \geq 0$ in $\bar{Q}_{T'}$ for all $T' < T$. By continuity, $u \geq 0$ in $\bar{Q}_{T}$.
\end{proof}

\begin{theorem}\label{MaximumPrincipleClassicalSol2}
Let $\Omega \subset \mathbb{R}^{n}$ be a bounded domain, $T > 0$ and let $u$ be a function with the same regularity as in
Theorem \ref{MaximumPrincipleClassicalSol1} and Dirichlet (zero) exterior conditions. Then we have the following two assertions
\begin{enumerate}
  \item If $\partial_{t}^{\alpha}(u-u_{0}) + (-\Delta)^{\beta}u \leq 0$ in $\Omega$, $t \in [0,T]$, then
  $\max_{\bar{Q}_{T}}u = \max_{\partial_{p}Q_{T}}u$.
  \item If $\partial_{t}^{\alpha}(u-u_{0}) + (-\Delta)^{\beta}u \geq 0$ in $\Omega$, $t \in [0,T]$, then
  $\min_{\bar{Q}_{T}}u = \min_{\partial_{p}Q_{T}}u$.
\end{enumerate}
\end{theorem}
\begin{proof}
We only prove the second result, the first one could be proved similarly.
If $u(x,0) \geq 0$, then we use Theorem \ref{MaximumPrincipleClassicalSol1} to see $u \geq 0$ in $\bar{Q}_{T}$,
and since $\partial_{p}Q_{T}\subset \bar{Q}_{T}$ and $u|_{\partial_{p}Q_{T}} = 0$, $\min_{\bar{Q}_{T}}u = \min_{\partial_{p}Q_{T}}u = 0$.
Otherwise, we assume that $u \geq 0$ not hold everywhere in $Q_{T}$, so there exists $(x_{0},t_{0}) \in \bar{Q}_{T}$
such that $\min_{\bar{Q}_{T}}u = u(x_{0},t_{0}) < 0$. By the proof of Theorem \ref{MaximumPrincipleClassicalSol1},
it is not possible that there exists a negative minimum in $Q_{T}\cup (\Omega\times\{T\})$, therefore,
the minimum in $\bar{Q}_{T}$ must be in $\partial_{p}Q_{T}$.
\end{proof}

\section{Maximum Principle for Weak Supersolution}

For convenience, denote $H_{e}^{s}(\Omega)$ ($s\in \mathbb{R}$) as follow
\begin{align}\label{DefineHExteriorSobolev1}
H_{e}^{s}(\Omega) := \left\{ u \in W^{s,2}(\mathbb{R}^{N}) \, : \, u = 0 \text{ in }\mathbb{R}^{N}\backslash\Omega \right\},
\end{align}
and $L_{e}^{p}(\Omega)$ ($1 \leq p \leq \infty$) as
\begin{align}\label{DefineLExteroriLegesgue1}
L_{e}^{p}(\Omega) := \left\{ u \in L_{e}^{p}(\mathbb{R}^{N}) \, : \, u = 0 \text{ in }\mathbb{R}^{N}\backslash\Omega \right\}.
\end{align}
Denote
\begin{align}\label{SecondForm1}
a(u,v) := \frac{c_{N,\beta}}{2}\int_{\mathbb{R}^{N}}\int_{\mathbb{R}^{N}}\frac{(u(x,t)-u(y,t))(\eta(x,t)-\eta(y,t))}{|x-y|^{N+2\beta}}dxdy.
\end{align}
We say that a function $u$ is a weak supersolution of (\ref{TimeSpaceEquation1}) in $Q_{T}$ with $f \in L^{\infty}(Q_{T})$
and $u_{0} \in L_{e}^{2}(\Omega)$, if $u$ belongs to the space
\begin{align*}
\begin{split}
V_{p} := & \Big\{ u \in L^{2p}([0,T];L_{e}^{2}(\Omega)) \cap L^{2}([0,T];H_{e}^{\beta}(\Omega)) \\
& \text{ such that }g_{1-\alpha}*(u-u_{0}) \in C([0,T];L_{e}^{2}(\Omega)), \text{ and }(g_{1-\alpha}*(u-u_{0}))|_{t=0} = 0 \Big\},
\end{split}
\end{align*}
and for any nonnegative test function
\begin{align}\label{TestFun1}
\eta \in H_{e}^{1,\beta}(Q_{T}) := W^{1,2}([0,T];L_{e}^{2}(\Omega)) \cap L^{2}([0,T];H_{e}^{\beta}(\Omega))
\end{align}
with $\eta|_{t=T} = 0$ there holds
\begin{align}\label{WeakFormulation1}
\begin{split}
\int_{0}^{T}\int_{\Omega}-\eta_{t}\left[ g_{1-\alpha}*(u-u_{0}) \right]dxdt
+ \int_{0}^{T}a(u,\eta)dt
\geq \int_{0}^{T}\int_{\Omega}f\eta dxdt.
\end{split}
\end{align}
We could provide an equivalent weak formulation of (\ref{TimeSpaceEquation1}) where kernel $g_{1-\alpha}$ is replaced
by a more regular kernel $g_{1-\alpha,m}$($m\in\mathbb{N}$).
For the detailed definition of $g_{1-\alpha,m}$, we refer to Section 2 in \cite{Zacher2008137}.
We could also introduce a function $h_{m}$ which satisfy $g_{1-\alpha,m} = g_{1-\alpha}*h_{m}$ with ``$*$'' represents the convolution operator.
For concisely, we only provide some important properties of functions $g_{1-\alpha,m}$ and $h_{m}$ as follows
\begin{align}\label{propertiesApp}
\begin{split}
& g_{1-\alpha,m} \in W^{1,1}([0,T]), \quad g_{1-\alpha,m} \rightarrow g_{1-\alpha} \text{ in }L^{1}([0,T]) \text{ as }m\rightarrow \infty, \\
& g_{1-\alpha,m} \text{ and }h_{m} \text{ are all nonnegative functions for every }m \in \mathbb{N}, \\
& \text{If }f\in L^{p}([0,T],X), 1\leq p<\infty, \text{ there holds }h_{m}*f \rightarrow f \text{ in }L^{p}([0,T],X),
\end{split}
\end{align}
where $X$ represents a Banach space.
Now we could show another definition of weak solution which is equivalent to equation (\ref{WeakFormulation1}).
\begin{lemma}\label{WeakEquivalent1}
Let $u \in V_{p}$ is a weak supersolution of equation (\ref{TimeSpaceEquation1}) if and only if for any nonnegative function
$\psi \in H_{e}^{\beta}(\Omega)$ one has
\begin{align}\label{WeakFormulation2}
\begin{split}
&\int_{\Omega}\psi \partial_{t}\left[ g_{1-\alpha,m}*(u-u_{0}) \right]dx
+ a(h_{m}*u,\psi)   \\
&\quad\quad\quad\quad\quad\quad\quad\quad\quad\quad\,\,\,\,
\geq \int_{\Omega}(h_{m}*f)\psi dx \, \text{ a.e.}\,\,t\in (0,T), \, m\in\mathbb{N}.
\end{split}
\end{align}
\end{lemma}
\begin{proof}
The `if' part is readily seen as follows. Given an arbitrary nonnegative $\eta \in H_{e}^{1,\beta}(Q_{T})$ satisfying $\eta|_{t=T} = 0$,
we take in (\ref{WeakFormulation2}) $\psi(x) = \eta(t,x)$ for any fixed $t\in (0,T)$, integrate from $t = 0$ to $t = T$,
and integrate by parts with respect to the time variable. Then by using the approximating properties of the kernels $h_{m}$,
we obtain (\ref{WeakFormulation1}).
To show the `only-if' part, we choose the test function
\begin{align}\label{TestFun2}
\eta(x,t) = \int_{t}^{T}h_{m}(\sigma - t)\varphi(\sigma,x)d\sigma = \int_{0}^{T-t}h_{m}(\sigma)\varphi(\sigma+t,x)d\sigma,
\end{align}
with arbitrary $m\in\mathbb{N}$ and nonnegative $\varphi\in H_{e}^{1,\beta}(Q_{T})$ satisfying $\varphi|_{t=T} = 0$;
$\eta$ is a nonnegative since $\varphi$ and $h_{m}$ are both nonnegative functions.
For the first term in (\ref{WeakFormulation1}), it can be transformed to
\begin{align}\label{EquivalentTran1}
\int_{0}^{T}\int_{\Omega}-\varphi_{t}\left[ g_{1-\alpha,m}*(u-u_{0}) \right] dxdt,
\end{align}
where we used $g_{1-\alpha,m} = g_{1-\alpha}*h_{m}$ and the Fubini's theorem.
For the term $\int_{0}^{T}a(u,\eta)dt$, we have
\begin{align*}
\begin{split}
&\int_{0}^{T}a(u,\eta)dt \\
= & \frac{c_{N,\beta}}{2}\int_{0}^{T}\int_{\mathbb{R}^{N}\times\mathbb{R}^{N}}\int_{t}^{T}
h_{m}(\sigma-t)\frac{(u(x,t)-u(y,t))(\varphi(x,\sigma)-\varphi(y,\sigma))}{|x-y|^{N+2\beta}}d\sigma dx dy dt \\
= & \frac{c_{N,\beta}}{2}\int_{0}^{T}\int_{\mathbb{R}^{N}}\int_{\mathbb{R}^{N}}
\frac{((h_{m}*u)(x,t)-(h_{m}*u)(y,t))(\varphi(x,t)-\varphi(y,t))}{|x-y|^{N+2\beta}}dxdydt \\
= & \int_{0}^{T}a(h_{m}*u,\varphi)dt.
\end{split}
\end{align*}
Observe that $g_{1-\alpha,m}*(u-u_{0}) \in {_{0}}W{^{1,2}}([0,T];L_{e}^{2}(\Omega))$ where $0$ means vanishing at $t=0$. Therefore,
combining (\ref{EquivalentTran1}) and the above equation, then integrating by parts and using $\varphi|_{t = T} = 0$ yields
\begin{align}\label{WeakFormWithT}
\int_{0}^{T}\int_{\Omega}\varphi \partial_{t}\left[ g_{1-\alpha,m}*(u-u_{0}) \right]dx
+ a(h_{m}*u,\varphi)dt \geq \int_{0}^{T}\int_{\Omega}(h_{m}*f)\varphi dxdt,
\end{align}
for all $m \in \mathbb{N}$ and $\varphi \in H_{e}^{1,\beta}(Q_{T})$ with $\varphi|_{t=T} = 0$.
By means of a simple approximation argument, we obtain that (\ref{WeakFormWithT}) holds true for any $\varphi$
of the form $\varphi(x,t) = \chi_{(t_{1},t_{2})}\psi(x)$ where $\chi_{(t_{1},t_{2})}$ denotes the characteristic
function of the time-interval $(t_{1},t_{2})$, $0<t_{1}<t_{2}<T$ and $\psi \in H_{e}^{\beta}(\Omega)$ is nonnegative.
Appealing to the Lebesgue's differentiation theorem \cite{grafakos2014classical},
the proof is complete.
\end{proof}

Now, we prove the maximum principle for the weak supersolution of (\ref{TimeSpaceEquation1}).
\begin{theorem}\label{WeakMaximumParabolic1}
Let $\Omega \subset \mathbb{R}^{N}$ be a bounded domain, $T > 0$, and $u$ a weak supersolution of problem (\ref{TimeSpaceEquation1})
with $u_{0} \geq 0$ a.e. in $\Omega$ and $f \geq 0$ a.e. in $\Omega \times [0,T]$. Then $u \geq 0$ a.e. in $\mathbb{R}^{N}\times[0,T]$.
\end{theorem}
\begin{proof}
We proceed by a contradiction argument. Taking $\varphi$ in (\ref{WeakFormWithT}) to be $u^{-}$, the negative part of $u$.
Suppose $u^{-}$ is nonzero in a set of positive measure. We know that
\begin{align}\label{WeakMaxProof1}
\begin{split}
\int_{0}^{T}\int_{\Omega}u^{-}\partial_{t}\left[ k_{m}*(u-u_{0}) \right]dx
+ & a(h_{m}*u,u^{-})\, dt \\
& \geq \int_{0}^{T}\int_{\Omega}(h_{m}*f)u^{-} dxdt.
\end{split}
\end{align}
Let us first analyze the second term on the left hand side of (\ref{WeakMaxProof1}).
Because $h_{m}*u \rightarrow u$ in $L^{2}([0,T];L_{e}^{2}(\Omega))$ as $m\rightarrow \infty$, we could deduce that
$\int_{0}^{T}a(h_{m}*u,u^{-})\,dt \rightarrow \int_{0}^{T}a(u,u^{-})\,dt \quad \text{as }m\rightarrow \infty.$
From
\begin{align*}
&\quad\quad\quad\quad
\int_{0}^{T}a(u,u^{-})dt = \int_{0}^{T}a(u^{+},u^{-})dt - \int_{0}^{T}a(u^{-},u^{-})dt, \\
&\int_{0}^{T}a(u^{-},u^{-})dt = \frac{c_{N,\beta}}{2}\int_{0}^{T}\int_{\mathbb{R}^{N}}\int_{\mathbb{R}^{N}}
\frac{(u^{-}(x,t) - u^{-}(y,t))^{2}}{|x-y|^{N+2\beta}}dxdydt > 0,
\end{align*}
we find that
\begin{align*}
\int_{0}^{T}a(u,u^{-})dt < \int_{0}^{T}a(u^{+},u^{-})dt.
\end{align*}
Noticing that $(u^{+}(x,t) - u^{+}(y,t))(u^{-}(x,t) - u^{-}(y,t)) \leq 0$, we obtain
\begin{align}\label{WeakMaxProof2}
\int_{0}^{T}a(u,u^{-})dt < \int_{0}^{T}a(u^{+},u^{-})dt \leq 0.
\end{align}
Hence, there exists a large positive number $M > 0$ such that if $m \geq M$, we have
\begin{align}\label{WeakMaxProof3}
\int_{0}^{T}a(h_{m}*u,u^{-})dt < 0.
\end{align}
For the first term on the left hand side of (\ref{WeakMaxProof1}), we have
\begin{align*}
\int_{0}^{T}\int_{\Omega}u^{-}\partial_{t} & \left[ g_{1-\alpha,m}*(u-u_{0}) \right]dxdt \\
& = \int_{0}^{T}\int_{\Omega}u^{-}\partial_{t}\left[ g_{1-\alpha,m}*u\right]dxdt
- \int_{0}^{T}\int_{\Omega}u^{-}g_{1-\alpha,m}u_{0}dxdt.
\end{align*}
Noticing that the second term on the righthand side is bigger than or equal to zero, we infer that
\begin{align}\label{WeakMaxProof4}
\int_{0}^{T}\int_{\Omega}u^{-}\partial_{t}\left[ g_{1-\alpha,m}*(u-u_{0}) \right]dxdt
\leq \int_{0}^{T}\int_{\Omega}u^{-}\partial_{t}\left[ g_{1-\alpha,m}*u\right]dxdt.
\end{align}
Using formula (\ref{AppendexTimeFractional3}), we obtain
\begin{align}\label{WeakMaxProof5}
\int_{0}^{T}\int_{\Omega}u^{-}\partial_{t}\left[ g_{1-\alpha,m}*u\right]dxdt \leq
-\frac{1}{2}\int_{\Omega}(g_{1-\alpha,m}*(u^{-})^{2})(x,T)dx \leq 0.
\end{align}
From (\ref{WeakMaxProof4}) and (\ref{WeakMaxProof5}), we conclude that
\begin{align}\label{WeakMaxProof6}
\int_{0}^{T}\int_{\Omega}u^{-}\partial_{t}\left[ g_{1-\alpha,m}*(u-u_{0}) \right]dxdt \leq 0 \quad \text{for }m\in \mathbb{N}.
\end{align}
Considering (\ref{WeakMaxProof3}) and (\ref{WeakMaxProof6}), for sufficiently large $m$, we deduce that
\begin{align}\label{WeakMaxProof7}
\begin{split}
\int_{0}^{T}\int_{\Omega}u^{-}\partial_{t}\left[ g_{1-\alpha,m}*(u-u_{0}) \right]dx
+ a(h_{m}*u,u^{-})\, dt < 0
\end{split}
\end{align}
Since $f \geq 0$ a.e. on $Q_{T}$, $u^{-} \geq 0$ a.e. on $Q_{T}$ and $g_{1-\alpha,m} \geq 0$ on $(0,T)$, we obtain
\begin{align*}
\int_{0}^{T}\int_{\Omega}(h_{m}*f)u^{-} dxdt \geq 0,
\end{align*}
which contradicts to (\ref{WeakMaxProof1}) and (\ref{WeakMaxProof7}).
Therefore, $u \geq 0$ a.e. in $\mathbb{R}^{N}\times[0,T]$.
\end{proof}

\section{Acknowledgements}

This work was partially supported by the National Natural Science Foundation of China under grant no. 11501439 and
the postdoctoral science foundation project of China under grant no. 2015M580826.


\bibliographystyle{siamplain}
\bibliography{reference}

\end{document}